\documentclass[11pt]{article}
\usepackage{geometry}
\usepackage[english]{babel}
\usepackage{longtable, tabularx, booktabs}
\usepackage{amsmath, amsthm, amsfonts, amscd,epsfig,lineno,amssymb, bbm, empheq}
\usepackage{enumitem, cellspace, float, lscape, extarrows}
\usepackage{graphicx, hyperref, cleveref}
\usepackage{fancyhdr, setspace, verbatim, babel}
\usepackage{multirow}
\usepackage{cite,epic}
\usepackage{adjustbox,makecell,float}
\usepackage[table]{xcolor}

\newtheorem{thm}{Theorem}[section]
\newtheorem{lemma}{Lemma}[section]

\newtheorem{con}{Construction}[section]

\crefname{lemma}{Lemma}{Lemmas}
\crefname{defi}{Definition}{Definitions}
\crefname{col}{Corollary}{Corollaries}
\crefname{thm}{Theorem}{Theorems}
\crefname{example}{Example}{Examples}
\crefname{problem}{Problem}{Problems}
\crefname{table}{Table}{Tables}
\crefname{pro}{Proposition}{Propositions}

\newcolumntype{M}{>{$}c<{$}}
\newcolumntype{N}{>{$}X<{$}}

\usepackage{caption}
\newcount\refno
\usepackage{cite}
\numberwithin{equation}{section}
\title{Symmetric Quantum Latin Squares with Maximal Cardinality}
\author{\footnotesize Yiwen Zhang, Haitao Cao\\
\scriptsize School of Mathematical Sciences, Ministry of Education Key Laboratory for NSLSCS,\\
\scriptsize Nanjing Normal University, Nanjing 210023, China}
\date{}
\begin{document}
\maketitle
\begin{abstract}
In this paper, we prove that  there exist a symmetric quantum Latin square with maximal cardinality of order $n$ for all $n\geq 6$, and a symmetric idempotent quantum Latin square with maximal cardinality of order $4k+2$ for all $k\geq 1$.
\end{abstract}

\section{Introduction}

Quantum Latin squares are a generalization of classical Latin squares in quantum field, introduced by Musto and Vicary \cite{2016}. Quantum Latin squares have significant applications and connections to mutually unbiased bases\cite{B}, quantum codes\cite{D} and $k$-uniform states\cite{Y}, etc.

A {\it quantum Latin square} of order $n$, denoted $QLS(n)$, is an $n \times n$ array whose entries are vectors in $\mathcal{H}_n$, such that the entries in each row and each column form an orthonormal basis of $\mathcal{H}_n$. A $\text{QLS}(n)$ can be obtained by replacing each entry $i\in [n]=\{0,1,\dots,n-1\}$ in a classical Latin square with the computational basis vector $|i \rangle \in \mathcal{H}_n$, where $|i \rangle$ is a unit column vector with its $(i +1)$-th component equal to $1$.
Let  $L(i,j)$ denote the entry in the $i$-th row and $j$-th column of a quantum Latin square $L$. If  $L(i,j)=L(j,i)$ for all $i,j\in [n]$, then $L$ is said to be {\it symmetric}, denoted by $SQLS$.

Suppose $|u\rangle$, $|v\rangle\in\mathcal{H}_n$ are unit vectors. $|u\rangle$ and $|v\rangle$ are {\it identical}, denoted by $|u\rangle=|v\rangle$, if there exists $\theta\in [0,2\pi)$ such that $|u\rangle=e^{\mathrm{i}\theta}|v\rangle$, where $\mathrm{i}$ is the imaginary unit. Otherwise, they are   {\it distinct} ($|u\rangle\not= |v\rangle$). The {\it cardinality} of a $QLS(n)$ $L$, denoted by $C(L)$, is the number of distinct entries in $L$. Obviously, $C(L)\in [n,n^2]\setminus \{n+1\}$. If $L$ consists of the computational basis $\{|0\rangle,|1\rangle,\dots,|n-1\rangle\}$ then $L$ is called {\it classical}. If $L$ is a quantum Latin square and $C(L)=n^2$, then $L$ has {\it maximal cardinality}, denoted by $MCQLS$. A {\it transversal} of $L$ refers to a set of $n$ elements, each located in a distinct row and a distinct column, forming an orthonormal basis of $\mathcal{H}_n$.

 Zhang et al. \cite{Z,Z2,tian} established the existence of an $MCQLS(n)$ for all $n\geq 4$. For more results on all possible cardinalities of quantum Latin squares, see \cite{P,J1,J2,ZYW,Z1}.

\begin{thm}\rm{(\cite{Z,Z2,tian})}\label{thm:1.0}
There exists an $MCQLS(n)$ for all $n\geq 4$.
\end{thm}

 For a symmetric quantum Latin square $L$, the maximum possible cardinality is $\frac{n(n+1)}{2}$. Therefore if  $C(L)=\frac{n(n+1)}{2}$, then $L$ has maximal cardinality, denoted by $MCSQLS$. It is easy to prove that every $SQLS(n)$ $L$ with $n\le 5$ satisfies $C(L)=n$. If the main diagonal of a $QLS(n)$ $L$ is a transversal, then $L$ is said to be {\it idempotent}. An idempotent $MCSQLS$ is denoted by $MCISQLS$. In this paper we establish the following existence theorems.

\begin{thm}\label{thm:1.1}
There exists an $MCSQLS(n)$ for all $n\geq 6$.
\end{thm}

\begin{thm}\label{thm:1.2}
There exists an $MCISQLS(n)$ for any $n\equiv 2\pmod{4}$ and $n\geq 6$.
\end{thm}

\section{Constructions from Hadamard matrices and conference matrices}

A {\it Hadamard matrix} of order $n$, denoted by $H$, is an $n\times n$ matrix with unimodular entries satisfying $HH^\dagger=nE_n$ where $H^\dagger$ is the conjugate transpose of $H$, $E_n$ is the identity matrix of order $n$. If the entries of $H$ in the first row and column are all $1$, then $H$ is called {\it dephased}. For a catalogue of complex Hadamard matrices, readers can refer to the online database https://chaos.if.uj.edu.pl/karol/hadamard/index.html and \cite{H1}; for recent research of Hadamard matrices, see \cite{H2} and \cite{H3}.


\begin{con}\rm{(\cite{MF})}\label{con:2.1}
Let $M = (m_{ij})$ and $F=(f_{ij})$, $i,j\in[n]$, be two Hadamard matrices of order $n$. Let
\begin{equation}\label{eq:12}
L(i,j)=\frac{1}{\sqrt{n}}
\begin{pmatrix}
m_{0i}f_{0j}\\
m_{1i}f_{1j}\\
\vdots\\
m_{(n-1)i}f_{(n-1)j}
\end{pmatrix}.
\end{equation}
Then $L$ is a $QLS(n)$. Furthermore if $M=F$, $L$ is symmetric.
\end{con}

\begin{lemma}\label{lemma:2.1}
If there exists a Hadamard matrix $H = (h_{ij})$ of order $n$ satisfying the following condition: for all $0\le i_1\le j_1\le n-1$, $0\le i_2\le j_2\le n-1$, there exist $r_1\neq r_2$, such that
\begin{equation}\label{eq:2}
\frac{h_{r_1i_1}h_{r_1j_1}}{h_{r_1i_2}h_{r_1j_2}}\neq\frac{h_{r_2i_1}h_{r_2j_1}}{h_{r_2i_2}h_{r_2j_2}},
\end{equation}
then there exists an $MCSQLS(n)$.
In particular, if $H$ is dephased, then the condition above can be simplified to: for all $0\le i_1\le j_1\le n-1$, $0\le i_2\le j_2\le n-1$, there exists $r$ such that
\begin{equation}\label{eq:3}
h_{ri_1}h_{rj_1}\neq h_{ri_2}h_{rj_2}.
\end{equation}
\end{lemma}

\begin{proof}
Apply \cref{con:2.1} with $M=F=H$ to obtain an $SQLS(n)$ $L$. Now we show that any two entries of $L$ are distinct. Otherwise, $L(i_1,j_1)=L(i_2,j_2)$, then $h_{ki_1}h_{kj_1}=ch_{ki_2}h_{kj_2}$ holds for all $k\in [n] $, where $c$ is a unit complex number, which implies $\frac{h_{ki_1}h_{kj_1}}{h_{ki_2}h_{kj_2}}=c$. This contradicts (2.2). In particular, if $H$ is dephased, then $c=1$ and it contradicts  (\ref{eq:3}). So $L$ has maximal cardinality. 
\end{proof}

\begin{lemma}\label{lemma:2.2}
There exists an $MCSQLS(8)$.
\end{lemma}
\begin{proof}
Let $a=\frac{1-\rm{i}2\sqrt{2}}{3}$, and \[
H=\begin{pmatrix}
1&1&1&1&1&1&1&1\\
1&-a^2&-1&a^2&a&-a&-a&a\\
1&-1&-1&-a&\frac{1}{a}&1&-\frac{1}{a}&a\\
1&a^2&-a&a&-1&a&-a&-a^2\\
1&a&\frac{1}{a}&-1&1&-1&\frac{1}{a}&-a\\
1&-a&1&a&-1&\frac{1}{a}&-\frac{1}{a}&-1\\
1&-a&-\frac{1}{a}&-a&-\frac{1}{a}&-\frac{1}{a}&1&-a\\
1&a&a&-a^2&-a&-1&-a&a^2
\end{pmatrix}.
\]

By \cref{lemma:2.1} it suffices to show \cref{eq:3} for all $0\le i_1\le j_1\le 7$, $0\le i_2\le j_2\le 7$. Because $3a^2-2a+3=0$, the minimal polynomial of $a$ over $\mathbb Q$ has degree $2$. Since it is not monic, $a$ is not an algebraic
integer. While every root of unity is an algebraic integer, so $a$ is not a root of
unity. And if there exists integer $k_1,k_2$ such that $a^{k_1}=\pm a^{k_2}$, then $k_1=k_2$. $h_{ij}$ has the form $\pm1 a^{e_{ij}}$, $e_{ij}\in\{-1,0,1,2\}$, $i,j\in [8]$. Thus $h_{ri_1}h_{rj_1}\neq h_{ri_2}h_{rj_2}$ forces $e_{ri_1}+e_{rj_1}\neq e_{ri_2}+e_{rj_2}$. Choose $r=1,2,3,5$ and we show for all $0\le i_1\le j_1\le 7$, $0\le i_2\le j_2\le 7$, $(e_{1i_1}+e_{1j_1},e_{2i_1}+e_{2j_1},e_{3i_1}+e_{3j_1},e_{5i_1}+e_{5j_1})^T\neq (e_{1i_2}+e_{1j_2},e_{2i_2}+e_{2j_2},e_{3i_2}+e_{3j_2},e_{5i_2}+e_{5j_2})^T$. Denote $E_{ij}=(e_{1i}+e_{1j},e_{2i}+e_{2j},e_{3i}+e_{3j},e_{5i}+e_{5j})$. Each $E_{ij}$ are listed in Table 1.

\begin{table}[h]
\centering
\caption{ The 36 distinct exponent pair sums $E_{ij}$ for coordinates $\{1,2,3,5\}$}.
\label{tab:sums}
\renewcommand{\arraystretch}{1.15}
\begin{tabular}{llll}
\toprule
$E_{00}=(0,0,0,0)$&$E_{07}=(1,1,2,0)$&$E_{01}=(2,0,2,1)$& $E_{34}=(3,0,1,1)$   \\
$E_{02}=(0,0,1,0)$&$E_{27}=(1,1,3,0)$ &$E_{67}=(2,0,3,-1)$ &$E_{36}=(3,0,2,0)$\\
$E_{22}=(0,0,2,0)$&$E_{44}=(2,-2,0,0)$ &$E_{12}=(2,0,3,1)$  &$E_{15}=(3,0,3,0)$\\
$E_{04}=(1,-1,0,0)$  & $E_{46}=(2,-2,1,-1)$  &$E_{03}=(2,1,1,1)$  &$E_{35}=(3,1,2,0)$ \\
$E_{06}=(1,-1,1,-1)$ & $E_{66}=(2,-2,2,-2)$  &$E_{23}=(2,1,2,1)$  &$E_{17}=(3,1,4,1)$ \\
$E_{24}=(1,-1,1,0)$  &$E_{45}=(2,-1,1,-1)$   &$E_{57}=(2,1,3,-1)$  &$E_{37}=(3,2,3,1) $ \\
$E_{26}=(1,-1,2,-1)$ & $E_{56}=(2,-1,2,-2)$   &$E_{77}=(2,2,4,0)$  & $E_{11}=(4,0,4,2)$\\
$E_{05}=(1,0,1,-1)$  &$E_{55}=(2,0,2,-2)$      &$E_{14}=(3,-1,2,1)$   & $E_{13}=(4,1,3,2)$\\
$E_{25}=(1,0,2,-1)$  &$E_{47}=(2,0,2,0)$    &$E_{16}=(3,-1,3,0) $   & $E_{33}=(4,2,2,2)$\\
\bottomrule
\end{tabular}
\end{table}
It is easy to verify every $E_{ij}$ is different. The proof is complete.
\end{proof}


\begin{thm}\rm{(\cite{Bjorck})}\label{thm:2.1}
(i) Let $p\equiv 3\pmod{4}$ be a prime and consider  the decomposition of the all-one matrix $J=I+S+N$ where $S$ and $N$ are $(0,1)$ circulant matrices whose first row is the characteristic vector of the quadratic residues (QR) and nonresidues (QNR) modulo $p$, respectively. Then the matrix $H:=I+S+\alpha N$ is a complex Hadamard matrix where $I$ is the identity matrix and $\alpha=\frac{-(p-1)\pm \mathrm{i}2\sqrt{p}}{p+1}$;\\
(ii) Let $p\equiv 1\pmod{4}$ be a prime and $S$ and $N$ are $(0,1)$ circulant matrices whose first row is the characteristic vector of the quadratic residues and nonresidues modulo $p$, respectively. Then the matrix $H:=I+\alpha S+\overline{\alpha} N$ is a complex Hadamard matrix where $\alpha=\frac{-1\pm \sqrt{p}+\mathrm{i}\sqrt{p^2-3p\pm 2\sqrt{p}}}{p-1}$.
\end{thm}

For a vector $v$ we write $v_k$ for its $(k+1)$-th coordinate. The following lemma is immediate from the definition of identical entries.
\begin{lemma}\label{lemma:d}
If two entries $\alpha,\beta\in \mathcal{H}_n$ have different numbers of distinct coordinate values or there exist $i\neq j\in [n]$ such that $\alpha_i=0$, $\beta_j=0$, then $\alpha\neq \beta$.
\end{lemma}

\begin{thm}\label{thm:2.2}
There exists an $MCSQLS(p)$ for any prime $p>5$.
\end{thm}
\begin{proof}
Let $\chi_p(x)$ be the quadratic character on $\mathbb{Z}_p$, then
\[\chi_p(x)=\begin{cases}
1, &\text{$x\in QR(p)$},\\
0, &\text{$x=0$},\\
-1, &\text{$x\in QNR(p)$}.
\end{cases}
\]
{\bf Case 1.} For $p\equiv 3\pmod{4}$, by \cref{thm:2.1} (i) we construct a Hadamard matrix $H=(h_{ij})$, where \[h_{ij}=\begin{cases}
1, &\text{$j-i\notin QNR(p)$},\\
\alpha, &\text{$j-i\in QNR(p)$}.
\end{cases}\]
Apply \cref{con:2.1} with $M=F=H$ to obtain an $SQLS(n)$ $L$, where \[L(i,i)=\frac{1}{\sqrt{n}}\begin{pmatrix}
h_{0i}^2 \\
\vdots \\
h_{(p-1)i}^2
\end{pmatrix},~ ~ L(i,j)=\frac{1}{\sqrt{n}}\begin{pmatrix}
h_{0i}h_{0j}\\
\vdots\\
h_{(p-1)i}h_{(p-1)j}
\end{pmatrix},  i\neq j.\] Thus
\[
h_{ki}^2=\begin{cases}
1, &\text{$i-k\notin QNR(p)$}, \\
\alpha^2, &\text{$i-k\in QNR(p)$},
\end{cases}~ ~ h_{ki}h_{kj}=\begin{cases}
1, &\text{$i-k,j-k\notin QNR(p)$},\\
\alpha, &\text{exactly one of $i-k,j-k\in QNR(p)$},\\
\alpha^2, &\text{$i-k,j-k\in QNR(p)$}.
\end{cases}
,  i\neq j.
\]

Let $$a=\#\{k|i-k,j-k\in QR(p),i\not=j,k\in \mathbb{Z}_p\},$$ $$b=\#\{k|i-k,j-k\in QNR(p),i\not=j,k\in \mathbb{Z}_p\},$$ $$c=\#\{k|i-k\in QR(p),j-k\in QNR(p),i\not=j,k\in \mathbb{Z}_p\},$$ $$d=\#\{k|i-k\in QNR(p),j-k\in QR(p),i\not=j,k\in \mathbb{Z}_p\}.$$ Note that $0\notin QR(p)\cup QNR(p)$, the cases $k=i$ and $k=j$ are excluded. Thus \begin{equation}\label{eq:2.5}
a+b+c+d=p-2.
\end{equation}

By $\sum\limits_{k\in \mathbb{Z}_p}\chi_p(i-k)\chi_p(j-k)=-1$ and $\chi_p(0)=0$,
\begin{equation}\label{eq:2.6}
a+b-c-d=-1.
\end{equation}

Then because $\#\{k|i-k\in QR(p)\}=\frac{p-1}{2}$, and the case $k=j$ is excluded from $a+c$, hence if $i-j\in QR(p)$, $a+c=\frac{p-3}{2}$; if $i-j\in QNR(p)$, $a+c=\frac{p-1}{2}$. Consequently,
\begin{equation}\label{eq:2.7}
a+c=\frac{p-1}{2}-\frac{1+\chi_p(i-j)}{2}.
\end{equation}

Similarly, we have
\begin{equation}\label{eq:2.8}
b+c=\frac{p-3}{2}+\frac{1+\chi_p(j-i)}{2}.
\end{equation}

Solve \cref{eq:2.5,eq:2.6,eq:2.7,eq:2.8}, we obtain
\[
a=b=\frac{p-3}{4},
c=\frac{p-1-2\chi_p(i-j)}{4},
d=\frac{p-1+2\chi_p(i-j)}{4}.
\]

We first show that the diagonal entries are pairwise distinct. If $L(i,i)=L(j,j)$ for $i\neq j$, then because $a>0$, there exists $k'$ such that $i-k',j-k'\notin QNR(p)$, namely $L(i,i)_{k'}=L(j,j)_{k'}=1$. Therefore $L(i,i)_{k}=L(j,j)_{k}$ for all $k\in [p]$ where $L(i,i)_{k}$ denotes $(k+1)$th-coordinate of $L(i,i)$, i.e., $h_{ki}^2=h_{kj}^2$. $h_{ki}=h_{kj}$ holds for all $k\in [p]$, a contradiction with the definition of Hadamard matrices. Thus $L(i,i)$ and $L(j,j)$ are distinct.

Then we prove any entry in main diagonal is distinct from any entry in off-diagonal. Because $p\geq 7$, a diagonal entry has exactly two distinct coordinate values, $1$ and $\alpha^2$, while an off-diagonal entry has three, $1,\alpha$ and $\alpha^2$. By \cref{lemma:d}, the entries in main diagonal are different from the off-diagonal entries.

It remains to show that two off-diagonal entries are distinct. By the symmetry of $L$ and the orthogonality of rows and columns, it suffices to show $L(i_1,j_1)$ and $L(i_2,j_2)$ are distinct for pairwise different $i_1,i_2,j_1,j_2$. Because $p\equiv 3\pmod{4}$, $\chi_p(-1)=-1$, exactly one of $i-j$ and $j-i$ lies in $QNR(p)$, an off-diagonal entry contains $b=\frac{p-3}{4}$ ``$\alpha^2$'' s, $c+d+1=\frac{p+1}{2}$ ``$\alpha$'' s, $p-\frac{p+1}{2}-\frac{p-3}{4}=\frac{p+1}{4}=a+1$ ``1'' s . Because $1,\alpha$ and $\alpha^2$ are pairwise different, if $L(i_1,j_1)=L(i_2,j_2)$, every ``1'' in $L(i_1,j_1)$ is supposed to be one of $1,\alpha$ and $\alpha^2$, the number of ``1'' in $L(i_1,j_1)$ is not equal to the number of ``$\alpha$'' and ``$\alpha^2$''. Therefore $L(i_1,j_1)_{k}=L(i_2,j_2)_{k}$ for all $k\in [p]$ if $L(i_1,j_1)=L(i_2,j_2)$. Because \[L(i_1,j_1)_k=\begin{cases}
\alpha^{1-\frac{\chi_p(i_1-k)+\chi_p(j_1-k)}{2}}, &\text{$k\neq j_1,i_1$},\\
\alpha^{\frac{1-\chi_p(i_1-j_1)}{2}}, &\text{$k=j_1$}\\
\alpha^{\frac{1-\chi_p(j_1-i_1)}{2}}, &\text{$k=i_1$}.
\end{cases}\]

If $L(i_1,j_1)_k=L(i_2,j_2)_k$ for any $k\in\{i_1,j_1,i_2,j_2\}$, then
\begin{equation*}
\left\{
\begin{aligned}
\chi_p(i_2-j_1)+\chi_p(j_2-j_1)-\chi_p(i_1-j_1)=1, \\
\chi_p(i_2-i_1)+\chi_p(j_2-i_1)-\chi_p(j_1-i_1)=1, \\
\chi_p(i_1-j_2)+\chi_p(j_1-j_2)-\chi_p(i_2-j_2)=1, \\
\chi_p(i_1-i_2)+\chi_p(j_1-i_2)-\chi_p(j_2-i_2)=1. \\
\end{aligned}
\right.
\end{equation*}

Adding $4$ equations above, because when $p\equiv 3\pmod{4}$, $\chi_p(-t)=-\chi_p(t)$ for all $t\in \mathbb{Z}_p$, $0=4$ a contradiction. Therefore $L$ is an $MCSQLS(p)$.

{\bf Case 2.} For $p\equiv 1\pmod{4}$, by \cref{thm:2.1} (ii) the Hadamard matrix $H$ is
\[h_{ij}=\begin{cases}
1, & \text{$i=j$},\\
\alpha, &\text{$j-i\in QR(p)$},\\
\overline{\alpha}, &\text{$j-i\in QNR(p)$}.
\end{cases}
h_{ki}^2=\begin{cases}
1,& \text{$k=i$},\\
\alpha^2, &\text{$i-k\in QR(p)$},\\
\overline{\alpha^2},& \text{$i-k\in QNR(p)$}.
\end{cases}\]

Thus

\[h_{ki}h_{kj}=\begin{cases}
1,& \text{$(i-k)(j-k)\in QNR(p)$},\\
\alpha,& \text{$i-j,j-i\in QR(p)$},\\
\overline{\alpha}, &\text{$i-j,j-i\in QNR(p)$},\\
\alpha^2,& \text{$i-k,j-k\in QR(p)$},\\
\overline{\alpha^2},& \text{$i-k,j-k\in QNR(p)$}.
\end{cases},i\neq j.\]

Similar to  Case 1, we may obtain \[
a=\frac{p-3-2\chi_p(j-i)}{4},
b=\frac{p-3+2\chi_p(j-i)}{4},
c=d=\frac{p-1}{4}.
\]

First we show two entries in main diagonal are distinct. A diagonal entry contains the value one ``1'', $\frac{p-1}{2}$ ``$\alpha^2$'' s, $\frac{p-1}{2}$ ``$\overline{\alpha^2}$'' s. Because $\alpha^4\neq 1$, if $L(i,i)=L(j,j)$, $L(i,i)_k=L(j,j)_k$ for all $k\in [p]$, a contradiction with definition of Hadamard matrices.

Next, as in (i), a diagonal entry and an off-diagonal entry have different numbers of distinct coordinate values when $p>5$, and are distinct by \cref{lemma:d}.

Finally it suffices to show two entries in off-diagonal entries are distinct. An off-diagonal entry contains the $c+d=\frac{p-1}{2}$ ``1'' s, $a$ ``$\alpha^2$'' s, $b$ ``$\overline{\alpha^2}$'' s, two ``$\alpha$'' s or ``$\overline{\alpha}$'' s, according to $i-j$ whether in $QR(p)$. When $p>5$, the number of ``1'' is $\frac{p-1}{2}$, while the number of every other values at most $\frac{p-1}{4}$. Therefore when $p>5$, the number of ``1'' is different from the number of every other coordinates and $\alpha,\overline{\alpha},\alpha^2,\overline{\alpha^2}\neq 1$, thus if $L(i_1,j_1)=L(i_2,j_2)$ for pairwise different $i_1,j_1,i_2,j_2$, $L(i_1,j_1)_k=L(i_2,j_2)_k$, i.e., $\alpha^{\chi_p(i_1-k)+\chi_p(j_1-k)}=\alpha^{\chi_p(i_2-k)+\chi_p(j_2-k)}$. Thus $\chi_p(i_1-k)+\chi_p(j_1-k)=\chi_p(i_2-k)+\chi_p(j_2-k)$ for all $k\in[p]$. Let $k=i_1$, $\chi_p(j_1-i_1)=\chi_p(i_2-i_1)+\chi_p(j_2-i_1)$. Because $i_1,j_1,i_2,j_2$ are pairwise different, $\chi_p(j_1-i_1),\chi_p(i_2-i_1),\chi_p(j_2-i_1)\in \{\pm1\}$, $\pm1\pm1\in\{\pm2,0\}$ not $\pm1$, a contradiction.
\end{proof}

A complex {\it conference matrix} of order $n$, denoted by $C$, is an $n\times n$ matrix with zero diagonal and unimodular off-diagonal entries  satisfying $CC^\dagger=(n-1)E_n$. If the entries of $C$ in the first row and column are all $1$, then $C$ is called {\it normalized}. On recent research of complex conference matrices, readers can refer to \cite{C1,C4,C3}.

\begin{lemma}\rm{(\cite{C1,C2})}\label{lemma:2.4}
There exists a symmetric complex conference matrix of order $n$ where $n$ is 
(i) $p^k+1$ where $p$ is a prime and $p^k\geq 4$;
(ii) $2t-1\equiv 1 \pmod 4$ is an odd prime power for $t\geq 3$.
\end{lemma}

\begin{con}\label{con:2.3}
If there exists a symmetric complex conference matrix of order $n$, then there exists an $MCSQLS(n+1)$.
\end{con}
\begin{proof}
Let $C=(c_{ij})$ be the symmetric complex conference matrix of order $n$.
Let $L(i,j)=\frac{1}{\sqrt{n-1}}\begin{pmatrix}
c_{0i}c_{0j}\\
c_{1i}c_{1j}\\
\vdots \\
c_{(n-1)i}c_{(n-1)j}\\
c_{ij}
\end{pmatrix}$ for all $i,j\in [n]$, $L(i,n)=|i\rangle\in\mathcal{H}_{n+1}$ for $i\in [n+1]$ and  $L(n,j)=|j\rangle\in\mathcal{H}_{n+1}$ for $j\in [n]$.

Because $C$ is symmetric and $c_{ki}c_{kj}=c_{kj}c_{ki}$, $k\in \{1,2,\dots,n\}$, $L$ is symmetric.

For fixed $i$ and $j\ne j'$, $i,j,j'\in [n]$, then $(L(i,j),L(i,j'))=\sum\limits_{k=0}^{n-1} |c_{ki}|^2 \overline{c_{kj}} c_{kj'}+\overline{c_{ij}}c_{ij'}=\sum\limits_{k=0}^{n-1}  \overline{c_{kj}} c_{kj'}=0$, and because $c_{ii}=0$, $L(i,j)$ is orthogonal to $|i\rangle$, thus every row of $L$ is orthogonal. By symmetry of $L$, $L$ is an $SQLS(n+1)$.

Now we prove $L$ has maximal cardinality. It's sufficient to prove that the zero components are mutually distinct for each entry in the upper triangular. It is clear that $L(i,j),i,j\le n$ are different from $|0\rangle,|1\rangle,\dots,|n-1\rangle$. For $L(i,i),i\le n$, the zero components are the $(i+1)$-th and the $(n+1)$-th coordinates. For $L(i,j),i\neq j$, the zero components are $(i+1)$-th and $(j+1)$-th coordinates. By \cref{lemma:d}, the entries in the upper triangular of $L$ are different.
\end{proof}

\begin{lemma}\label{lemma:2.5}
There exists an $MCSQLS(n)$ for $n=p^k+2$ where $p^k\geq 4$ is a prime power, and for $n=2t$ where $t\geq 3$ and $2t-1\equiv 1 \pmod 4$ is an odd prime power.
\end{lemma}

\begin{proof}
Apply \cref{lemma:2.4} and \cref{con:2.3} to get the conclusion.
\end{proof}

\section{MCISQLS from Hadamard matrices}

\begin{thm}\label{thm:H}\rm{(\cite{butson})}
For any odd prime $p$, let $q=\frac{p-1}{2}$, $w=e^{2\pi i/p}$, and $n$ be the smallest quadratic nonresidue modulo $p$. Then there exists a Hadamard matrix $H=(h_{ij})$ of order $2p$, where

\[h_{ij}=
\begin{cases}
w^{i(qi+j)}, &i,j\in[p],\\
w^{ni(qi+(j-p))}, &i\in[p],j\in[2p]\setminus[p],\\
w^{-q(j-n(i-p))^2},&i\in[2p]\setminus[p],j\in[p],\\
w^{-nq(j-i)^2},&i,j\in[2p]\setminus[p].
\end{cases}
\]

\end{thm}

\begin{thm}\label{thm:2p}
There exists an  $MCISQLS(2p)$ for any odd prime $p$.
\end{thm}
\begin{proof}
Apply \cref{con:2.1} with $M=F=H$, where $H$ is the matrix in \cref{thm:H}. The entries of the resulting $SQLS$ $L$ fall into three cases. If $i,j\in[p]$, \[
L(i,j)_k=\begin{cases}
w^{k(i+j)+k^2(p-1)}, &k\in [p],\\
w^{-q\left[(i-n(k-p))^2+(j-n(k-p))^2\right]}, &k\in [2p]\setminus [p] .
\end{cases}
\]

If $i\in[2p]\setminus [p]$, $j\in [p]$,\[
L(i,j)_k=\begin{cases}
w^{k(qk+i-p)+nk(qk+j)}, &k\in [p],\\
w^{-q((i-p)-n(k-p))^2-nq(j-(k-p))^2}, &k\in [2p]\setminus [p] .
\end{cases}
\]

If $i,j\in[2p]\setminus [p]$,\[
L(i,j)_k=\begin{cases}
w^{nk(2qk+i+j-2p)}, &k\in [p],\\
w^{-nq\left[(j-k-p)^2+(i-k-p)^2 \right]}, &k\in [2p]\setminus [p] .
\end{cases}
\]

Note that $L(i,j)_0=1$ for all $i,j\in[2p]$. To prove $L$ is maximal cardinality, it suffices to show it is impossible for all $k\geq 1$ such that $L(i_1,j_1)_k= L(i_2,j_2)_k$ for pairwise different $i_1\le j_1,i_2\le j_2$.

\begin{enumerate}

\item[Case 1:] $i_1,i_2,j_1,j_2\in [p]$.

If $k\in[p]$, $k(i_1+j_1)+k^2\equiv k(i_2+j_2)+k^2\pmod{p}$. If $k\in [2p]\setminus[p]$, $-q(i_1^2+j_1^2)-2qn(k-p)(i_1+j_1)\equiv -q(i_2^2+j_2^2)-2qn(k-p)(i_2+j_2)$. Because $p$ is an odd prime, if $L(i_1,j_1)_k=L(i_2,j_2)_k$ for all $k\in [2p]$,\[
i_1+j_1\equiv i_2+j_2\pmod{p},\\
i_1^2+j_1^2\equiv i_2^2+j_2^2\pmod{p}.
\]
should be satisfied, a contradiction. The two cases $i_1,j_1,i_2,j_2\in [2p]\setminus[p]$ and $i_1,i_2\in[p]$, $j_1,j_2\in [2p]\setminus[p]$ are similar.

\item[Case 2:] $i_1,j_1\in [p], i_2,j_2\in[2p]\setminus [p]$.

If $L(i_1,j_1)_k=L(i_2,j_2)_k$ for $k\in [p]$, then $(-n+1)k+n(i_2+j_2-2p)-(i_1+j_1)\equiv 0\pmod{p}$. Because $-n+1\equiv 0\pmod{p}$, $k=\frac{n(i_2+j_2-2p)-(i_1+j_1)}{-n+1}$, at most one $k$ such that $L(i_1,j_1)_k=L(i_2,j_2)_k$. The remaining two cases $i_2\in [p]$, $i_1,j_1,j_2\in[2p]\setminus[p]$ and $i_1,i_2,j_1\in[p]$, $j_2\in [2p]\setminus [p]$ are similar.
\end{enumerate}

Thus $L$ is an $MCSQLS(2p)$. Then we prove $L$ is idempotent, namely $L(i,i)$ is orthogonal to $L(j,j)$ for all $0\le i<j\le 2p-1$.

\begin{enumerate}
\item[Case 1:] $i<j\in [p]$.

Then $(L(i,i),L(j,j))=\sum\limits_{k\in \mathbb{Z}_p} w^{2(i-j)k}+\sum\limits_{t\in \mathbb{Z}_p} w^{(p-1)(j-i)(-2nt+j+i)}=0$ where $t=k-p$ because $p$ is an odd prime. The case $i<j\in [2p]\setminus [p]$ is similar.

\item[Case 2:] $i\in [p],j\in [2p]\setminus [p]$.

Let $G(a)=\sum\limits_{k\in \mathbb{Z}_p} w^{ak^2}$. It is known that $G(ab)=\chi_p(a)G(b)$ when $a\neq 0$. Then
\begin{align*}
(L(i,i),L(j,j))&=\sum\limits_{k\in \mathbb{Z}_p}w^{(n-1)k^2+2k(i-nj)}+w^{2q(nj^2-i^2)}\sum\limits_{t\in \mathbb{Z}_p} w^{-2q\left[(n-n^2)k^2-2nk(i-j)\right]}\\
&=w^{-(n-1)^{-1}(i-nj)^2}\sum\limits_{k\in \mathbb{Z}_p} w^{(n-1)(k+(n-1)^{-1}(i-nj))^2}\\
&+w^{2q(nj^2-i^2)+n(i-j)^2(1-n)^{-1}}\sum\limits_{t\in \mathbb{Z}_p} w^{-n(1-n)(t-(i-j)(1-n))^2}\\
&=w^{-(n-1)^{-1}(i-nj)^2} G(n-1)\\
 &+w^{(n-1)^{-1}\left[-(n-1)(nj^2-i^2)-n(i-j)^2\right]} G(n(n-1))\\
&=w^{-(n-1)^{-1}(i-nj)^2} G(n-1)-w^{(n-1)^{-1}\left[-i^2-n^2j^2+2nij \right]} G(n-1)\\
&=0
\end{align*}
\end{enumerate}
Therefore $L$ is an  $MCISQLS(2p)$.
\end{proof}

\section{Direct Product Construction}
In this section, we give a direct product construction. For brevity, we use $A\cap B$ to denote the set of identical entries both in $QLS$ $A$ and $QLS$ $B$. We first recall the following lemma on the existence of infinitely pairwise disjoint $QLS$s.

\begin{lemma}\rm{(\cite{Z1})}\label{lemma:3.1}
Suppose $n\geq 3$. Let $A$ and $B$ be two $QLS(n)$s with $C(A)=c_1$ and $C(B)=c_2$, respectively. Then for any positive integer $k$, there exist $QLS(n)$s $A_1,A_2,\dots,A_k$ of cardinality $c_1$ and $B_1,B_2,\dots,B_k$ of cardinality $c_2$, such that $S\cap T=\emptyset$ for any distinct $S,T\in\{A_1,\dots,A_k,B_1,\dots,B_k\}$.
\end{lemma}

\begin{lemma}\label{lemma:U}
Let $L=(|l_{ij}\rangle)$ be an $SQLS(n)$, $U$ be a unitary matrix, then $UL=(U|l_{ij}\rangle)$ is also an $SQLS(n)$ and $C(L)=C(UL)$.
\end{lemma}

By \cref{lemma:U}, \cref{lemma:3.1} can be generalized to $SQLS$.
\begin{lemma}\label{lemma:3.2}
Suppose $n\geq 3$. Let $A$ and $B$ be two $SQLS(n)$s with $C(A)=c_1$ and $C(B)=c_2$, respectively. Then for any positive integer $k$, there exist $SQLS(n)$s $A_1,A_2,\dots,A_k$ of cardinality $c_1$ and $B_1,B_2,\dots,B_k$ of cardinality $c_2$, such that $S\cap T=\emptyset$ for any distinct $S,T\in\{A_1,\dots,A_k,B_1,\dots,B_k\}$.
\end{lemma}

\begin{lemma}\label{lemma:3.3}
If there exists an $MCSQLS(m)$, then for every positive integer $h$ there exist $h$ $MCQLS(m)$s and $2$ $MCSQLS(m)$s such that all entries appearing in these $h+2$ $QLS(m)$s are mutually distinct.
\end{lemma}
\begin{proof}
By \cref{lemma:3.2}, there exist $2$ disjoint $MCSQLS(m)$s. By \cref{lemma:3.1}, there exist $m(m+1)+h$ pairwise disjoint $MCQLS(m)$s. Among the two $MCSQLS(m)$s, there are exactly $m(m+1)$ different entries. By the pigeonhole principle, there exist $h$ pairwise disjoint $MCQLS(m)$s such that they are disjoint from two $MCSQLS(m)$s.
\end{proof}

\begin{lemma}\label{lemma:3.4}
For any $n\geq 2$, if there exists an $MCSQLS(m)$, then there exists an $MCSQLS\allowbreak(nm)$. Furthermore for every odd $n\ge 3$, if there exists an  $MCISQLS(m)$, then there exists an  $MCISQLS(nm)$.
\end{lemma}
\begin{proof}
Let $A=(|a_{i,j}\rangle)$ be a classical $SQLS(n)$. If $n$ is odd, there exists a symmetric idempotent Latin square of order $n$, which means $a_{i,i}=i$. If $n$ is even, there exists a symmetric half-idempotent Latin square, which means $a_{i,i}=i$ when $i< \frac{n}{2}$, $a_{i,i}=i-\frac{n}{2}$ when $i\geq \frac{n}{2}$ \cite{S}.

By \cref{lemma:3.3}, let $B_1$, $B_2$ be two disjoint $MCSQLS(m)$s, $C_{i,j}$ be $\frac{n(n-1)}{2}$ $MCQLS(m)$s, for all $0<j<i\le n-1$ such that $B_1$, $B_2$ and $C_{i,j}$ for all $0<j<i\le n-1$ are pairwise disjoint. Then we construct a matrix of order $nm$, denoted by $L$. $L$ can be divided into $n^2$ blocks, each of size $m\times m$. The $(i,j)-th$ block, denoted by $L_{i,j}$, $i,j=0,1,\dots,n-1$, is defined as follows:
\begin{equation*}
L_{i,j}=\begin{cases}
|a_{i,j}\rangle \otimes B_1, & i=j,i< \lfloor\frac{n}{2} \rfloor, \\
|a_{i,j}\rangle \otimes B_2, & i=j,i\geq \lfloor\frac{n}{2} \rfloor, \\
|a_{i,j}\rangle \otimes C_{i,j}, &i>j, \\
|a_{i,j}\rangle \otimes C_{j,i}^T, &j>i.
\end{cases}
\end{equation*}

The diagonal blocks are symmetric since $B_1,B_2$ are symmetric. Then it suffices to show $L_{i,j}(k,l)$ and $L_{j,i}(l,k)$, $0<j<i\le n-1$, $k,l\in [m]$. $L_{i,j}(k,l)=|a_{i,j}\rangle\otimes C_{i,j}(k,l)$, $L_{j,i}(l,k)=|a_{j,i}\rangle\otimes C_{i,j}^T(l,k)=|a_{i,j}\rangle\otimes C_{i,j}(k,l)=L_{i,j}(k,l)$. Therefore $L$ is symmetric.

Then we prove $L$ is row-orthogonal. If two entries $\alpha$ and $\beta$ in the same row of $L$ are in different blocks $L_{i,j}$ and $L_{i,j'}$, $\alpha$ and $\beta$ are orthogonal since $|a_{ij}\rangle$ and $|a_{ij'}\rangle$ are orthogonal; if $\alpha$ and $\beta$ are in the same block, suppose $\alpha=|a_{i,j}\rangle\otimes u$, $\beta=|a_{i,j}\rangle\otimes v$, then $u$ and $v$ are in the same row of a $QLS(m)$. Since $u,v$ are orthogonal, $\alpha$ and $\beta$ are orthogonal. Therefore $L$ is row-orthogonal. By symmetry of $L$, $L$ is column-orthogonal. Consequently, $L$ is an $SQLS(nm)$.

$L$ is maximal cardinality is also clear. Suppose two different lower triangular entries $\alpha=|a_{i,j}\rangle\otimes u$, $\beta=|a_{i',j'}\rangle\otimes v$. If $|a_{i,j}\rangle\neq |a_{i',j'}\rangle$, then $\alpha$ and $\beta$ are distinct clearly. If $|a_{i,j}\rangle=|a_{i',j'}\rangle$, we can suppose $(i,j)\neq (i',j')$ because $u,v$ come from the same $QLS(m)$, while $B_1,B_2,C_{i,j}$, $0<j<i\le n-1$ are maximal cardinality, $u,v$ are distinct. Thus $\alpha$ and $\beta$ are distinct. If $(i,j)\neq (i',j')$,
\begin{enumerate}
\item[Case 1:] $i=j$, $i'=j'$. Because $B_1$ and $B_2$ are disjoint, two entries in main diagonal are distinct.
\item[Case 2:] $i\neq j$, $i'\neq j'$. Because $C_{i,j}$ and $C_{i',j'}$ are disjoint, two entries in off-diagonal are distinct.
\item[Case 3:] $i=j$, $i'\neq j'$. $B_1$ or $B_2$ is disjoint from $C_{i,j}$, $0<j<i\le n-1$, the entry in main diagonal is distinct from the entry in off-diagonal.
\end{enumerate}
Therefore $L$ is an $MCSQLS(nm)$.

Furthermore, Let $n\ge 3$ be odd, take $A$ idempotent and $B_1$ idempotent, and use
$|a_{i,i}\rangle\otimes B_1$ for every diagonal block, the other blocks being as above. $B_1$ and $C_{i,j}$, $0<i<j\le n-1$ are pairwise disjoint. $L$ is still an $MCSQLS(nm)$ similarly.

We only show the main diagonal of $L$ is a transversal, namely any two different entries $\alpha$ and $\beta$ in main diagonal are orthogonal.  If $\alpha$ and $\beta$ are in different blocks, since $L$ is idempotent, $\alpha$ and $\beta$ are orthogonal. Otherwise, $B$ is idempotent which guarantees they are also orthogonal.
\end{proof}

\section{Main result}

In this section, we prove our main result. 

\noindent \textbf{Proof of \cref{thm:1.1}:} There exists an $MCSQLS(n)$ for $n\in\{6,8,9,10,15,25\}$ where $8$ comes from \cref{lemma:2.2} and $\{6,9,10,15,25\}$ come from \cref{lemma:2.5}.
For all other values of $n\geq 6$, let $n=2^i3^j5^kp_1^{l_1}p_2^{l_2}\dots p_s^{l_s}$ where $p_1,p_2,\dots,p_s$ are primes not less than $7$, $i,j,k,l_1,l_2$, $\dots,l_s$ are nonnegative integers. We distinguish the following three cases.

\begin{enumerate}

\item[Case 1:] $i=j=k=0$. Then there exists $t$ with $l_t>0$. By \cref{thm:2.2,lemma:3.4}, we have done.

\item[Case 2:] Exactly two of $i,j,k$ are zero. If $i=j=0$, there exists an $MCSQLS(25)$, by \cref{lemma:3.4} with $m=25$, $n=5^{k-2}$, there exists an $MCSQLS(5^k)$ for $k\geq 2$. If $i=k=0$, there exists an $MCSQLS(9)$, thus there exists an $MCSQLS(3^j)$ for $j\geq 2$ by \cref{lemma:3.4} with $m=9$, $n=3^{j-2}$. If $j=k=0$, there exists an $MCSQLS(8)$, thus there exists an $MCSQLS(2^i)$ for $i\geq 3$ by \cref{lemma:3.4} with $m=8$, $n=2^{i-3}$.

\item[Case 3:] At most one of $i,j,k$ is zero. If $k=0$, there exists an $MCSQLS(6)$, thus there exists an $MCSQLS(2^i3^j)$ for $i,j\geq 1$ by \cref{lemma:3.4} with $m=6$, $n=2^{i-1}3^{j-1}$. If $j=0$, there exists an $MCSQLS(10)$, then there exists an $MCSQLS(2^i5^k)$ for $i,k\geq 1$ by \cref{lemma:3.4} with $m=10$, $n=2^{i-1}5^{k-1}$. If $i=0$, there exists an $MCSQLS(15)$, thus there exists an $MCSQLS(3^j5^k)$ for $j,k\geq 1$ by \cref{lemma:3.4} with $m=15$, $n=3^{j-1}5^{k-1}$.

\end{enumerate}

Combining Cases 1-3 and \cref{lemma:3.4} completes the proof.

\noindent \textbf{Proof of \cref{thm:1.2}:}
For any $n\equiv 2\pmod{4}$, $n\geq 6$, suppose $n=2k$ with $k\ge3$ odd. If $k$ is an odd prime number, \cref{thm:2p} gives an  $MCISQLS(n)$. If $k$ is an composite, suppose $k=k_1k_2$ with $k_1$ is an odd prime number and $k_2$ is an odd, an  $MCISQLS(2k_1)$ exists by \cref{thm:2p} then there exists an  $MCISQLS(n)$ by \cref{lemma:3.4} with $n=k_2$, $m=2k_1$.


\section*{Acknowledgements}
\noindent The authors are very grateful to Prof. Lie Zhu from Soochow University for his valuable suggestions.  H. Cao's research was supported by the National Natural Science Foundation of China (Grants No. 12471313 and No. 12071226).


\end{document}